\documentclass{article}

\usepackage{amsmath}
\usepackage{amsfonts}
\usepackage{amsthm}
\usepackage{amssymb}
\usepackage{mathtools}

\usepackage{cleveref}

\AtBeginDocument{%
\let\temp\phi
\let\phi\varphi
\let\varphi\temp
\let\temp\epsilon
\let\epsilon\varepsilon
\let\varepsilon\temp
}

\let\Re\relax
\DeclareMathOperator{\Re}{Re}
\let\Im\relax
\DeclareMathOperator{\Im}{Im}

\DeclareMathOperator{\interior}{int}

\newtheorem{proposition}{Proposition}
\newtheorem{corollary}[proposition]{Corollary}
\newtheorem{theorem}[proposition]{Theorem}

\theoremstyle{definition}
\newtheorem*{definition}{Definition}
\newtheorem*{remark}{Remark}

\usepackage{xcolor}

\usepackage[
	style=alphabetic
]{biblatex}
\begin{document}
\title{Chabauty Limits of Fermat Spirals}
\author{Yohay Ailon Tevet}
\date{\today}
\maketitle

\begin{abstract}
A Fermat spiral is a set of points of the form \(\sqrt{n}e^{2\pi i\alpha n}\) for \(\alpha \in \mathbb{R}\). In this paper we prove that the Chabauty limits of Fermat spirals are always closed subgroups of \(\mathbb{R}^2\), and conclude that no Fermat spirals are dense forests. Furthermore, we show that if \(\alpha\) is badly approximable the Chabauty limits are always lattices, for which we give a characterisation.
\end{abstract}

\section{Introduction and Main Results}
A \emph{spiral set} is a set of the form \(\{\sqrt{n}u_n\;|\;n\in\mathbb{N}\}\) such that \(u_n \in \mathbb{S}^1\). A spiral set is called a \emph{Fermat spiral} if \(u_n = e^{2\pi i\alpha n}\) for some \(\alpha\in\mathbb{R}\). Fermat spirals, and, in particular, the \emph{sunflower spiral} (the Fermat spiral with constant angle difference \(\phi = (1+\sqrt{5})/2\), the golden ratio), have long been studied. In \cite{Vogel_1979}, Vogel studied a close relative of the sunflower spiral that is found in nature (specifically, with angle difference \(\phi^2\)), and in \cite{Marklof_2020}, Marklof studied spiral sets where the angle is of the form \(\alpha\sqrt{n}\).

A number \(\alpha\notin\mathbb{Q}\) with continued fraction convergents \(p_j/q_j\) is called \emph{badly-approximable} if 
\[0 < c_\alpha \coloneqq \liminf q_j|q_j\alpha - p_j| .\]

A discrete set \(X \subseteq \mathbb{R}^d\) is a \emph{dense forest} if
\begin{align}\label{eq:density}
    \limsup_{r\to\infty} \frac{\#(B_r(0)\cap X)}{r^d} < \infty
\end{align}
and there is a decreasing function \(V\colon (0,\infty)\to(0,\infty)\), called a \emph{visibility function}, such that for every \(\epsilon > 0\), every \(\epsilon\times V(\epsilon)\) rectangle intersects \(X\). In \cite{Akiyama_2020}, Akiyama showed that the Fermat spirals with constant angle difference are Delone if and only if the angle difference is badly approximable, and in \cite{Adiceam_2022} Adiceam and Tsokanos studied the visibility properties of general spiral sets, questions that are related to the open Danzer problem (see \cite{Adiceam_2021}). In this paper we show that no Fermat spiral is a dense forest, which answers a question posed by Adiceam and Tsokanos whether the sunflower spiral is a dense forest. In fact, we prove a stronger result about Chabauty limits of Fermat spirals. We do this by identifying all the Chabauty limits of some Fermat spirals.

Note that in the case that \(X\) is a spiral set the \(\limsup\) in \cref{eq:density} may be replaced by \(\lim\), and the limit is always \(1\).

In order to state our results more precisely, we introduce some terminology.

\begin{definition}
The \emph{Chabauty-Fell topology} of \(\mathbb{R}^d\) is the topology, on the set of closed sets of \(\mathbb{R}^d\), given by the metric
\[d(A, B) = \min\{1, \Delta(A, B) \}\]
where 
\[\Delta(A, B) = \inf \{ \epsilon > 0 \;|\; B_{1/\epsilon}(0) \cap A \subseteq \mathcal{N}_\epsilon(B), 
    B_{1/\epsilon}(0) \cap B \subseteq \mathcal{N}_\epsilon(A)\}\]
and \(\mathcal{N}_\epsilon(A)\) is the \(\epsilon\)-neighbourhood of \(B\) in the Euclidean norm, that is
\[\mathcal{N}_\epsilon(A) = \bigcup_{x\in A} B_\epsilon(x) .\]
It is known that this is a compact topological space.

A \emph{Chabauty limit} of a closed set \(X\subseteq \mathbb{R}^d\) is the limit (if it exists) of \(T_nX\), where \(T_n\) are translations of \(\mathbb{R}^d\) such that \(T_n(0) \to \infty\).
\end{definition}

We state now the main result, which will be proved later.
\begin{theorem}
\label{thm:chabauty_limits_of_fermat_spiral}
The non-empty Chabauty limits of a Fermat spiral are translations of closed subgroups of \(\mathbb{R}^2\).
\end{theorem}

Intuitively, the Chabauty topology measures how close two closed sets are in a big neighbourhood around the origin. The Chabauty limits can likewise be thought of as how the set \(X\) looks as we go farther from the origin. It is easy to show that the only Chabauty limits of a Fermat spiral with a rational \(\alpha\) are the empty set and lines. In the case that \(\alpha\) is not rational and not badly-approximable little is known, however, the following theorem states that in the case where \(\alpha\) is badly-approximable the Chabauty limits can be well understood.

\begin{theorem}
\label{thm:catergorisation_of_badly_approximable_limits}
Let \(X\) be a Fermat spiral with badly-approximable angle difference \(\alpha\), and let \((\beta, c, \tilde{c})\) be the limit of some subsequence of
\[\bigl(q_{j+1}/q_j, q_j(q_j\alpha - p_j), q_{j+1}(q_{j+1}\alpha-p_{j+1}) \bigr)\]
where \(p_j/q_j\) are the convergents of \(\alpha\). The Chabauty limits of \(X\) are precisely translates of the lattices
\[\Lambda = \sqrt{\pi} R A_t \begin{bmatrix}
        1 & c \\ \beta & \tilde{c}/\beta
    \end{bmatrix} \mathbb{Z}^2\]
for every such a triplet \((\beta, c, \tilde{c})\), rotation matrix \(R\), and \(0 < t\), where \(A_t = \left[\begin{smallmatrix}t\\&1/t\end{smallmatrix}\right]\).
\end{theorem}

In the case of \(\alpha = \phi\) we have \(\beta = \phi\) and \(c = -\tilde{c} = \pm1/\sqrt{5}\), therefore the Chabauty limits of the sunflower spiral are of the form
\[\Lambda = \sqrt{\pi} R A_t \begin{bmatrix}
        1 & \pm1/\sqrt{5} \\ \phi & \mp1/(\sqrt{5}\phi)
    \end{bmatrix} \mathbb{Z}^2 .\]

\section{Applications of \Cref{thm:chabauty_limits_of_fermat_spiral}}
Throughout the paper we use the following characterisation of the Chabauty limit:

\begin{proposition}
Let \(X\subseteq \mathbb{R}^d\) be a closed set and \(T_n\) be translations such that \(T_nX\) converges to \(\tilde{X}\) in the Chabauty topology. For a point \(x\in\mathbb{R}^d\) we have that \(x\in\tilde{X}\) if and only if there is a sequence of points \(x_n\in T_nX\) that converges to \(x\).
\end{proposition}

\begin{proof}
Assume that \(x\in\tilde{X}\), then for any \(n\) denote by \(x_n\) the point closest to \(x\) in \(T_nX\) (there must be such a point since \(T_nX\) is closed). For every \(0 < \epsilon < 1\) and \(\|x\| < 1/\epsilon\), for some \(N\) and every \(N\leq n\), we have that \(d(X, T_nX) < \epsilon\), therefore \(x \in \mathcal{N}_\epsilon(T_nX)\) and so \(\|x_n - x\| < \epsilon\).

In the other direction, we have a sequence \(x_n \in T_nX\) that converges to \(x\), and denote by \(\tilde{x}_n \in \tilde{X}\) the point closest to \(x_n\). For every \(0 < \epsilon < 1\) there is \(N_1\) such that for every \(N_1 \leq n\) we have \(\|x-x_n\| < \epsilon\); similarly, there is \(N_2\) such that for every \(N_2 \leq n\) we have \(x_n \in \mathcal{N}_\epsilon(\tilde{X})\), therefore \(\|\tilde{x}_n-x_n\| < \epsilon\). For \(N = \max\{N_1, N_2\}\) and \(N\leq n\) we have
\[\|\tilde{x}_n - x\| \leq \|\tilde{x}_n - x_n\| + \|x_n - x\| < 2\epsilon\]
and so we have \(\tilde{x}_n \to x\) and since \(\tilde{X}\) is closed \(x\in\tilde{X}\).
\end{proof}

\begin{definition}
A discrete set \(X\subseteq \mathbb{R}^d\) is called \emph{relatively dense} if there is some \(r > 0\) such that for any \(x\in\mathbb{R}^d\) we have \(B_r(x) \cap X \neq\emptyset\); it is called \emph{uniformly discrete} if there is some \(s > 0\) such that for any \(x\in\mathbb{R}^d\) we have  \(|B_s(x) \cap X| \leq 1\); it is called \emph{Delone} if it is both relatively dense and uniformly discrete.
\end{definition}

\begin{proposition}
\label{prop:delone_chabauty}
If \(X\) is Delone, then every Chabauty limit of it is Delone.
\end{proposition}

\begin{proof}
Let \(X\) be a Delone set and \(T_n\) be translations such that \(T_nX\) converges to \(\tilde{X}\). There is some \(r\) such that \(X\) is \(r\)-relatively dense, therefore for any \(x\in\mathbb{R}^d\) and any translation \(T_n\) we have \(y_n \in B_r(x) \cap T_nX\), so the sequence \(y_n\) has a converging subsequence in \(\overline{B_r(x)}\), and the limit of it must be in \(\tilde{X}\), therefore \(\tilde{X}\cap B_{r+1}(x) \neq \emptyset\).

Similarly, there is some \(s\) such that \(X\) is \(s\)-uniformly discrete, therefore for any \(x\in\mathbb{R}^d\) and any translation \(T_n\) we have \(|B_s(x)\cap T_nX| \leq 1\), so it is clear that \(|B_s(x) \cap \tilde{X}| \leq 1\).
\end{proof}

As a consequence of \Cref{thm:chabauty_limits_of_fermat_spiral} we obtain:
\begin{corollary}
\label{cor:chabauty_limits_of_badly_approximable_fermat_spiral}
The Chabauty limits of a Fermat spiral with badly-approximable angle difference are translations of lattices.
\end{corollary}

\begin{proof}
There are exactly six closed subgroups of \(\mathbb{R}^2\) (up to isomorphism): \(0\), \(\mathbb{Z}\), \(\mathbb{R}\), \(\mathbb{Z}^2\), \(\mathbb{Z}\times\mathbb{R}\), \(\mathbb{R}^2\). In \cite{Akiyama_2020}, Akiyama showed that for a badly-approximable angle difference, the Fermat spiral is Delone, therefore the limits must be Delone as well, and the only Delone closed subgroups of \(\mathbb{R}^2\) are the lattices.
\end{proof}

\begin{proposition}
\label{prop:dense_forest_chabauty}
For a set \(X\subseteq \mathbb{R}^d\), the following are equivalent:
\begin{enumerate}
    \item \(X\) is not a dense forest;
    \item \(X\) has a Chabauty limit with an infinite strip with no points (i.e. there is an \(\epsilon\) and there is a line such that there are no points \(\epsilon\)-close to it);
    \item \(X\) has a Chabauty limit with a half-infinite strip with no points (i.e. there is an \(\epsilon\) and there is a ray such that there are no points \(\epsilon\)-close to it).
\end{enumerate}
\end{proposition}

\begin{proof}
For \((1\implies2)\), if \(X\) is not a dense forest, there is some \(\epsilon > 0\) such that for every \(V > 0\) there is \(\epsilon\times V\) rectangle that does not intersect with \(X\). Let \(T_n\) be the transformation that translates the centre of such a rectangle for \(V = n\) to the origin. Let \(\theta_n\) be the angle the long side of the rectangle makes with the horizontal direction, then we may assume that \(\theta_n\) converges by passing to a subsequence. We now may pass to a converging subsequence of \(T_nX\) since the Chabauty topology is compact, and the limit must have an infinite strip with no points.

The implication \((2\implies3)\) is trivial.

For \((3\implies1)\), assume there is such a Chabauty limit \(\tilde{X}\) with a half-infinite strip with no points of width \(\epsilon\), achieved by the translations \(T_n\). For every \(V > 0\) there is \(n\) such that \(d(\tilde{X}, T_nX) < \min\{\epsilon/2, 1/(2V)\}\), this means there is a \(V\times(\epsilon/2)\) rectangle that does not intersect \(T_nX\), so there must be such rectangle that does not intersect \(X\).
\end{proof}

\begin{corollary}
\label{cor:no_fermat_spiral_is_a_dense_forest}
No Fermat spiral is a dense forest.
\end{corollary}

\begin{proof}
Let \(\alpha\) be the angle difference. If \(\alpha\in \mathbb{Q}\) the points lie on a finite number of rays from the origin, and so for some \(\epsilon>0\) we can look at the half-infinite strip \([a,\infty)\times[\epsilon/2,\epsilon]\) and for large enough \(a\) it will not intersect any point. If \(\alpha\notin\mathbb{Q}\) is not badly-approximable, Akiyama showed in \cite{Akiyama_2020} that the Fermat spiral is not relatively dense, and if \(\alpha\notin\mathbb{Q}\) is badly-approximable, it follows from \Cref{cor:chabauty_limits_of_badly_approximable_fermat_spiral} and \Cref{prop:dense_forest_chabauty}.
\end{proof}

\begin{remark}
The author conjectures that for \(\alpha\) that is not badly-approximable the only non-empty Chabauty limits are isomorphic to \(\mathbb{R}\), \(\mathbb{Z}^2\), and \(\mathbb{Z}\times\mathbb{R}\), and in particular the limit cannot be \(\mathbb{R}^2\). If this is indeed the case, then we need not split the preceding proof into cases.
\end{remark}

\section{Proof of \Cref{thm:chabauty_limits_of_fermat_spiral}}
Throughout this section, we will represent points in the plane as both vectors and complex numbers.

Let \(X\) be a Fermat spiral and \(T_j\) be translations such that \(T_jX\) converges in the Chabauty topology to a non-empty set. Without loss of generality, we may assume that the limit has a point in the origin. Let \(v_1\) and \(v_2\) be two non-zero points in the limit, so we can have \(n_j, a_j, b_j\) such that \(T_jx_{n_j} \to 0, T_jx_{n_j+a_j} \to v_1\), and \(T_jx_{n_j+b_j} \to v_2\), so we have \(x_{n_j+a_j}-x_{n_j} \to v_1\) and \(x_{n_j+b_j}-x_{n_j} \to v_2\). Note that since \(T_j(0) \to \infty\) we have \(n_j\to\infty\). It is sufficient to show that \(x_{n_j+a_j+b_j}-x_{n_j} \to v_1+v_2\), or that 
\[x_{n_j+a_j+b_j}-x_{n_j+a_j} - (x_{n_j+b_j} - x_{n_j}) \to 0 .\]
Note that we can write 
\begin{align*}
    x_{t+b_j} - x_t 
    &= \sqrt{t+b_j}e^{2\pi i\alpha(t+b_j)} - \sqrt{t}e^{2\pi i\alpha t}\\
    &= \underbrace{(\sqrt{t+b_j}-\sqrt{t})e^{2\pi i\alpha(t+b_j)}}_{u_t^1} + 
        \underbrace{\sqrt{t}(e^{2\pi i\alpha(t+b_j)}-e^{2\pi i\alpha t})}_{u_t^2} .
\end{align*}
We want to show that \(u_{n_j+a_j}^i - u_{n_j}^i \to 0\) for \(i=1,2\).

For \(i=1\) we have that the angle of \(u_{n_j}^1\) is \(2\pi\alpha (n_j+b_j)\), while the angle of \(u_{n_j+a_j}^1\) is \(2\pi\alpha (n_j+a_j+b_j)\), so the difference in the angles is \(2\pi\alpha a_j\) (modulo \(2\pi\)) which must converge to \(0\) otherwise \(x_{n_j+a_j}-x_{n_j}\) will not converge. The norms are \(\sqrt{n_j+b_j}-\sqrt{n_j}\) and \(\sqrt{n_j+a_j+b_j}-\sqrt{n_j+a_j}\) respectively. Without loss of generality, assume \(|a_j| \leq |b_j|\) and \(|\sqrt{n_j+b_j}-\sqrt{n_j}| \leq c\), therefore
\[c^2 - 2c\sqrt{n_j} + n_j\leq n_j+b_j \leq c^2 + 2c\sqrt{n_j} + n_j\]
so we have \(b_j = O(\sqrt{n_j})\). Denote
\begin{align*}
    d_j
    &\coloneqq \sqrt{n_j+b_j}-\sqrt{n_j} - (\sqrt{n_j+a_j+b_j}-\sqrt{n_j+a_j}) \\
    &\leq \sqrt{n_j+|b_j|}-\sqrt{n_j} - \Bigl(\sqrt{n_j+2|b_j|}-\sqrt{n_j+|b_j|}\Bigr) \\
    &\leq \sqrt{n_j+c\sqrt{n_j}}-\sqrt{n_j} - \Bigl(\sqrt{n_j+2c\sqrt{n_j}}-\sqrt{n_j+c\sqrt{n_j}}\Bigr) .
\end{align*}
Note that \(\sqrt{x+c}-\sqrt{x} = c/(2\sqrt{x}) + O(1/x)\), so we have
\begin{align*}
    d_j
    &\leq {n_j}^{1/4}\biggl(\frac{c}{2}\cdot{n_j}^{-1/4} + O\bigl({n_j}^{-1/2}\bigr)\biggr) -
        {n_j}^{1/4}\biggl(\frac{c}{2}\cdot{n_j}^{-1/4} + O\bigl({n_j}^{-1/2}\bigr)\biggr) \\
    &= \frac{c}{2} - \frac{c}{2} + O\bigl({n_j}^{-1/4}\bigr) = o(1) .
\end{align*}
We can now write
\begin{align*}
    u^1_{n_j} - u^1_{n_j+a_j}
    &= (\sqrt{n_j+b_j}-\sqrt{n_j})e^{2\pi i\alpha(n_j+b_j)} - \\
    &\qquad (\sqrt{n_j+a_j+b_j}-\sqrt{n_j+a_j})e^{2\pi i\alpha(n_j+a_j+b_j)} \\
    &= \bigl((\sqrt{n_j+b_j}-\sqrt{n_j})(1-e^{2\pi i\alpha a_j}) + d_je^{2\pi i\alpha a_j}\bigr)e^{2\pi i\alpha(n_j+b_j)} ,
\end{align*}
and since \(\sqrt{n_j+b_j}-\sqrt{n_j}\) is bounded, \(e^{2\pi i\alpha a_j} \to 1\), and \(d_j \to 0\) we have that \(u^1_{n_j} - u^1_{n_j+a_j} \to 0\).

For \(i=2\) we have
\[u_t^2 = \sqrt{t}(e^{2\pi i\alpha(t+b_j)}-e^{2\pi i\alpha t}) = 
    \sqrt{t}e^{2\pi i\alpha t}(e^{2\pi i\alpha b_j} - 1) = x_t(e^{2\pi i\alpha b_j} - 1) ,\]
therefore
\begin{align*}
    \lim_{j\to\infty} u_{n_j+a_j}^2 - u_{n_j}^2
    &= \lim_{j\to\infty} (x_{n_j+a_j}-x_{n_j})(e^{2\pi i\alpha b_j} - 1) \\
    &= v_1 \lim_{j\to\infty} (e^{2\pi i\alpha b_j} - 1)
    = v_1(1-1) = 0 .
\end{align*}

To show that the Chabauty limit is closed under taking inverses, we can look at the points \(x_{n_j-a_j}\), and if \(x_{n_j-a_j}-x_{n_j}\) converges to \(v\), then from the argument above we have that \(v+v_1 = 0\), therefore \(-v_1 = v\) is in the Chabauty limit. Note that we have
\[\sqrt{n_j - a_j} - \sqrt{n_j} - (\sqrt{n_j} - \sqrt{n_j + a_j}) \to 0\]
since \(\sqrt{n_j + a_j} - \sqrt{n_j}\) converges, and this clearly implies that \(x_{n_j-a_j} - x_{n_j}\) converges. We have shown that the Chabauty limit is a subgroup, and by the definition of the Chabauty topology it must be closed.

\section{Properties of the Limit}
We now turn to studying the non-empty Chabauty limits of Fermat spirals as given by \Cref{thm:chabauty_limits_of_fermat_spiral} in the badly-approximable case. Let \(X\) be a Fermat spiral with a badly-approximable angle difference \(\alpha\) and let \(T_j\) be a sequence of translations such that \(T_jX\) converges in the Chabauty topology. Without loss of generality, we may assume that the limit has a point at the origin, so \(T_jX\) converges to a lattice \(\Lambda\). Let \(n_j\) be a sequence of indices such that \(T_jx_{n_j} \to 0\).

Let \(p_n/q_n\) be the continued fraction convergents of \(\alpha\). We can write \(q_{n+1} = a_nq_n + q_{n-1}\) for some positive natural numbers \(a_n\), which are the continued fraction expansion of \(\alpha=[a_1,a_2,\ldots]\). The number \(\alpha\) is badly-approximable if and only if the sequence \(a_n\) is bounded, therefore it is equivalent to the sequence \(q_{n+1}/q_n = a_n + q_{n-1}/q_n\) being bounded. Note we also have \(c_\alpha \leq q_n|q_n\alpha-p_n| \leq 1\), and so we can always pass to a converging subsequence. Two facts which will prove useful are \((q_n\alpha - p_n)(q_{n+1}\alpha - p_{n+1}) < 0\) and
\[q_n|q_{n+1}\alpha - p_{n+1}| + q_{n+1}|q_n\alpha - p_n| = 1 .\]
For a reference, see \cite{Cassels_1957}.

\begin{proposition}
\label{prop:closest_point}
The closest point to \(x_n \in X\) is \(x_{n-q}\) where \(p/q\) is some convergent of \(\alpha\).
\end{proposition}

\begin{proof}
Let \(a\) be the number such that \(x_{n+a}\) is the closest point to \(x_n\). Note that \(a\) is negative since otherwise the point \(x_{n-a}\) is closer. Let \(p/q\) be the convergent with largest denominator such that \(q \leq -a\), then we have that the difference in both the norm and the angle between \(x_{n-q}\) and \(x_n\) is less than the difference between \(x_{n+a}\) and \(x_n\).
\end{proof}

Let \(p_j/q_j\) be the convergent from \Cref{prop:closest_point} with respect to \(x_{n_j}\), and \(\tilde{p}_j/\tilde{q}_j\) the following convergent. Without loss of generality, we may pass to a subsequence such that the sequence
\[\bigl(\tilde{q}_j/q_j, q_j(q_j\alpha-p_j), \tilde{q}_j(\tilde{q}_j\alpha-\tilde{p}_j)\bigr)\]
converges to a point in \(\mathbb{R}^3\) and we will denote the limit by \((\beta, c, \tilde{c})\). We have \(\beta \geq 1\) and \(|c|, |\tilde{c}| \geq c_\alpha > 0\).

\begin{corollary}
\label{cor:shortest_vector}
One of the shortest non-zero vectors in \(\Lambda\) is the limit of a subsequence of \(x_{n_j+q_j} - x_{n_j}\).
\end{corollary}

\begin{proof}
The spiral set is relatively dense (see \cite{Akiyama_2020}), so there is \(0 < r\) such that for every \(j\) there is \(a_j\) such that \(\|x_{n_j} - x_{n_j+a_j}\| < r\), therefore, by the definition of \(q_j\), \(\|x_{n_j} - x_{n_j+q_j}\| < r\), and the sequence \(x_{n_j-q_j}-x_{n_j}\) must have a converging subsequence, say to a vector \(-v\neq 0\). Therefore, we have that a subsequence of  \(x_{n_j+q_j}-x_{n_j}\) converges to the vector \(v\). Let \(u\neq 0\) be some vector in the lattice, then there is a sequence \(a_j\) such that \(x_{n_j+a_j}-x_{n_j} \to u\), and by \Cref{prop:closest_point} for each \(j\) we have 
\[\|x_{n_j+a_j}-x_{n_j}\| \geq \|x_{n_j-q_j}-x_{n_j}\| ,\]
therefore \(\|u\| \geq \|-v\| = \|v\|\), so \(v\) must be one of the shortest non-zero vectors in \(\Lambda\).
\end{proof}

\begin{proposition}
\label{cor:following_convergent_vector_converges}
The sequence \(x_{n_j+\tilde{q}_j} - x_{n_j}\) has a converging subsequence.
\end{proposition}

\begin{proof}
It is sufficient to show it is bounded. Since \(\sqrt{n_j+q_j}-\sqrt{n_j}\) converges we have \(q_j = 2t\sqrt{n_j} + o(\sqrt{n_j})\) for some \(t\), therefore \(\tilde{q}_j = 2t\beta\sqrt{n_j} + o(\sqrt{n_j})\), meaning \(\sqrt{n_j+\tilde{q}_j}-\sqrt{n_j} = t\beta + o(1)\).
\end{proof}

We assume from now on that \(x_{n_j+q_j} - x_{n_j}\) and \(x_{n_j+\tilde{q}_j} - x_{n_j}\) converge to the vectors \(v\) and \(\tilde{v}\) respectively, by passing to a subsequence if necessary.

\begin{proposition}
\label{prop:convergent_vectors_span_the_limit_lattice}
\(\Lambda\) is generated by the vectors \(v\) and \(\tilde{v}\).
\end{proposition}

\begin{proof}
Denote by \(A\) the convex hull of \(\pm v\) and \(\pm \tilde{v}\). Assume that the lattice is not spanned by \(v\) and \(\tilde{v}\), so there is a vector \(u \in \Lambda\) that is not spanned by them, and let \(a_j\) be a sequence such that \(x_{n_j+a_j}-x_{n_j} \to u\). Without loss of generality, we may assume that \(u\) is in the basic parallelogram with the vertices at \(0\), \(v\), \(\tilde{v}\), and \(v+\tilde{v}\). If \(u \notin A\), then it is in the triangle with its vertices at \(v\), \(\tilde{v}\), and \(v+\tilde{v}\), therefore \(u-(v+\tilde{v}) \in A\). This means that without loss of generality we may assume that \(u \in A\) and we may also assume that \(a_j > 0\) because otherwise we can take \(x_{n_j-a_j} - x_{n_j} \to -u \in A\).

Since \(\|v\| \leq \|\tilde{v}\|\) and \(u \in A\setminus\{\pm v, \pm\tilde{v}\}\) we have \(\|u\| < \|\tilde{v}\|\), therefore, for \(j\) large enough we have
\[\|x_{n_j+a_j}\| \leq \|x_{n_j+a_j} - x_{n_j}\| + \|x_{n_j}\| <
    \|x_{n_j+\tilde{q}_j} - x_{n_j}\| + \|x_{n_j}\| \leq \sqrt{n_j+\tilde{q}_j} ,\]
therefore \(0 \leq a_j < \tilde{q}_j\). By the definition of convergents, we have \(|a_j\alpha - b_j| \geq |q_j\alpha - p_j|\) where \(b_j\) is the nearest integer to \(a_j\alpha\).

Denote by \(A_j\) the convex hull of \(x_{n_j\pm q_j}\) and \(x_{n_j\pm\tilde{q}_j}\), and by \(B_j\) the sector
\[B_j = \{x\in\mathbb{R}^2 \;|\; \text{the angle between \(x\) and \(x_{n_j}\) is at most \(2\pi|\alpha q_j-p_j|\)}\} .\]
Note that \(T_jA_j\to A\) in the Chabauty topology and \(A_j\setminus\{x_{n_j\pm q_j}\} \subseteq \interior B_j\), therefore we must conclude that for \(j\) large enough \(x_{n_j+a_j} \in \interior B_j\), meaning that the angle difference between \(x_{n_j+a_j}\) and \(x_{n_j}\) is strictly less than \(2\pi|\alpha q_j - p_j|\), however, it is clearly \(2\pi|\alpha a_j-b_j|\). We have \(|\alpha a_j-b_j| < |\alpha q_j - p_j|\), which is a contradiction.
\end{proof}

We now turn to prove \Cref{thm:catergorisation_of_badly_approximable_limits}. The proof will be presented in two parts, the first that every Chabauty limit of \(X\) is of the specified form, and later that every lattice of the specified form is indeed achieved.

\begin{proof}[Proof of \Cref{thm:catergorisation_of_badly_approximable_limits} (part 1)]
Assume the lattice \(\Lambda\) is a Chabauty limit of \(X\), then it is spanned by the vectors \(v\) and \(\tilde{v}\). Denote
\begin{align*}
    v_j^1 = (\sqrt{n_j+q_j} - \sqrt{n_j})e^{2\pi i\alpha n_j} &, \qquad
    v_j^2 = \sqrt{n_j+q_j}(e^{2\pi i\alpha (n_j+q_j)} - e^{2\pi i\alpha n_j}) \\
    \tilde{v}_j^1 = (\sqrt{n_j+\tilde{q}_j} - \sqrt{n_j})e^{2\pi i\alpha n_j} &, \qquad
    \tilde{v}_j^2 = \sqrt{n_j+\tilde{q}_j}(e^{2\pi i\alpha (n_j+\tilde{q}_j)} - e^{2\pi i\alpha n_j}) ,
\end{align*}
and we have \(v_j^1+v_j^2 \to v\) and \(\tilde{v}_j^1 + \tilde{v}_j^2 \to \tilde{v}\). Recall that we assumed
\[\bigl(\tilde{q}_j/q_j, q_j(q_j\alpha-p_j), \tilde{q}_j(\tilde{q}_j\alpha-\tilde{p}_j)\bigr) \to (\beta, c, \tilde{c}) .\]
We may also assume, by passing to a subsequence, that \(e^{2\pi i\alpha n_j}\) converges to \(e^{i\theta}\), and denote \(R = \left[\begin{smallmatrix}\cos(\theta)&-\sin(\theta)\\\sin(\theta)&\cos(\theta)\end{smallmatrix}\right]\), that is, the rotation matrix by \(\theta\). This means that \(v_j^1\) converges to \(te^{i\theta}\) for some \(t > 0\) (\(t \neq 0\) because \(X\) is Delone), therefore \(q_j = 2t\sqrt{n_j} + o(\sqrt{n_j})\) and \(\tilde{q}_j = 2\beta t\sqrt{n_j} + o(\sqrt{n_j})\), which means that \(\tilde{v}_j^1\) converges to \(\beta t e^{i\theta}\).

We can now compute \(v_j^2\), for which we have
\begin{align*}
    \lim_{j\to\infty} v_j^2 =
    \lim_{j\to\infty} e^{2\pi i\alpha n_j} \sqrt{n_j+q_j}(e^{2\pi i\alpha q_j} - 1) = 
    e^{i\theta}\cdot \lim_{j\to\infty} \sqrt{n_j+q_j}(e^{2\pi i\alpha q_j} - 1) .
\end{align*}
Looking at the real and imaginary parts, using \(q_j\alpha -p_j = c/q_j + o(1/q_j)\), we get:
\begin{align*}
    &\mkern-36mu \Re \sqrt{n_j+q_j}(e^{2\pi i\alpha q_j} - 1)
    = \sqrt{n_j+q_j}\cos(2\pi\alpha q_j) \\
    &= \sqrt{n_j+q_j} O\biggl(\frac{1}{{q_j}^2}\biggr)
    = \sqrt{\frac{n_j}{{q_j^2}}+\frac{1}{q_j}} O\biggl(\frac{1}{q_j}\biggr)\\
    &= \biggl(\frac{n_j}{q_j} + o(1)\biggr) O\biggl(\frac{1}{q_j}\biggr)
    = (2t + o(1)) o(1) = o(1) \\
    &\mkern-36mu \Im \sqrt{n_j+q_j}(e^{2\pi i\alpha q_j} - 1)
    = \sqrt{n_j+q_j}\sin(2\pi\alpha q_j) \\
    &= \sqrt{n_j+q_j} \biggl(2\pi(\alpha q_j-p_j) + O\biggl(\frac{1}{{q_j}^3}\biggr)\biggr) \\
    &= \sqrt{\frac{n_j}{{q_j}^2} + \frac{1}{q_j}} \biggl(2\pi c + o(1) + O\biggl(\frac{1}{{q_j}^2}\biggr)\biggr) \\
    &= \biggl(\frac{\sqrt{n_j}}{q_j} + o(1)\biggr)\cdot (2\pi c + o(1))
    = \frac{1}{2t}\cdot 2\pi c + o(1) = \frac{\pi c}{t} + o(1) ,
\end{align*}
which gives \(v_j^2 \to \pi c/t\cdot ie^{i\theta}\), and similarly we have \(\tilde{v}_j^2 \to \pi\tilde{c}/(\beta t)ie^{i\theta}\). This means that in vector notation we have
\begin{align*}
    v = \lim v_j^1+v_j^2 = R\begin{bmatrix}t \\ \pi c/t \end{bmatrix} &&
    \tilde{v} = \lim \tilde{v}_j^1+\tilde{v}_j^2 = R \begin{bmatrix}\beta t \\ \pi\tilde{c}/(\beta t)\end{bmatrix} ,
\end{align*}
therefore
\[\Lambda = R \begin{bmatrix}t & \beta t \\ \pi c/t & \pi\tilde{c}/(\beta t) \end{bmatrix}\mathbb{Z}^2 .\]
By dividing both rows by \(\sqrt{\pi}\) we get
\[\Lambda = \sqrt{\pi} R \begin{bmatrix}t/\sqrt{\pi} & \beta t/\sqrt{\pi} \\ 
    c \sqrt{\pi}/t & \tilde{c}/\beta\cdot\sqrt{\pi}/t \end{bmatrix}\mathbb{Z}^2 =
    \sqrt{\pi} R A_{t/\sqrt{\pi}}\begin{bmatrix}1 & \beta \\ c & \tilde{c}/\beta \end{bmatrix} \mathbb{Z}^2 .\]
This finishes the proof that all the Chabauty limits are of the specified form.
\end{proof}

\begin{definition}
The co-volume of the lattice \(v\mathbb{Z}+\tilde{v}\mathbb{Z}\) is the volume of the fundamental parallelogram of the lattice, that is \(|\det(v, \tilde{v})|\).
\end{definition}

\begin{proposition}
\label{prop:co_volume_of_limit_lattice}
\(\Lambda\) has co-volume of \(\pi\).
\end{proposition}

\begin{remark}
In the definition of a dense forest, we required that for a set \(X\) we have
\[\limsup_{r\to\infty} \frac{\# (B_r(0)\cap X)}{r^d} < \infty ,\]
which is called the \emph{asymptotic density} of \(X\). Note that in the case that \(X\) is a lattice, the asymptotic density is \(1\) exactly when the co-volume is \(\pi\). Essentially, what this proposition tells us is that the asymptotic density of a Fermat spiral with a badly-approximable angle difference has the same asymptotic density as its Chabauty limits.
\end{remark}

\begin{remark}
\Cref{prop:co_volume_of_limit_lattice} clearly follows from the main result is given in \cite{Yamagishi_2021}, where they prove that the area of Voronoi cells converges to \(\pi\) (see the paper for definitions). However, we give a brief proof for completeness.
\end{remark}

\begin{proof}
We know every Chabauty limit is a lattice of the form
\[\sqrt{\pi}RA_t \begin{bmatrix}1&\beta\\c&\tilde{c}/\beta\end{bmatrix} \mathbb{Z}^2 ,\]
therefore the co-volume of it is 
\[\left|\det\left(\sqrt{\pi}RA_t \begin{bmatrix}1&\beta\\c&\tilde{c}/\beta\end{bmatrix}\right)\right| = 
    \pi \left|\frac{\tilde{c}}{\beta} - \beta c \right| .\]
Note that we have 
\begin{align*}
    \frac{\tilde{c}}{\beta} - c\beta
    &= \bigl(\tilde{q}_j(\tilde{q}_j\alpha - \tilde{p}_j) + o(1)\bigr)\biggl(\frac{q_j}{\tilde{q}_j} + o(1)\biggr) - \\
    &\qquad \bigl(q_j(q_j\alpha - p_j) + o(1)\bigr)\biggl(\frac{\tilde{q}_j}{q_j} + o(1)\biggr) \\
    &= q_j(\tilde{q}_j\alpha - \tilde{p}_j) - \tilde{q}_j(q_j\alpha - p_j) + o(1) ,
\end{align*}
and since \(q_j\alpha - p_j\) and \(\tilde{q}_j\alpha - \tilde{p}_j\) have opposite signs we have
\begin{align*}
    \biggl|\frac{\tilde{c}}{\beta} - c\beta\biggr|
    &= |q_j(\tilde{q}_j\alpha - \tilde{p}_j) - \tilde{q}_j(q_j\alpha - p_j)| + o(1) \\
    &= q_j|\tilde{q}_j\alpha - \tilde{p}_j| + \tilde{q}_j|q_j\alpha - p_j| + o(1) = 1 + o(1) ,
\end{align*}
therefore, the co-volume is \(\pi\).
\end{proof}

We will show that all lattices of the specified form are indeed achieved as Chabauty limits.
\begin{proof}[Proof of \Cref{thm:catergorisation_of_badly_approximable_limits} (Part 2)]
Let \(p_j/q_j\) be some subsequence of the convergents of \(\alpha\) such that
\[\bigl(\tilde{q}_j/q_j, q_j(q_j\alpha-p_j), \tilde{q}_j(\tilde{q}_j\alpha-\tilde{p}_j)\bigr) \to (\beta, c, \tilde{c}) .\]
First, we will show that for each \(0 < t\) there is a limit which is the lattice \(\sqrt{\pi}R'A_{t/\sqrt{\pi}}\left[\begin{smallmatrix}1&\beta\\c&\tilde{c}/\beta\end{smallmatrix}\right]\), for some rotation matrix \(R'\); second, we will show that \(R'\) can be replaced by any rotation matrix.

Define \(n_j = aq_j\tilde{q}_j\) for \(a = 1/(4t^2\beta)\), and we can assume that \(e^{2\pi\alpha in_j}\) converges to some \(e^{i\theta'}\) by passing to a subsequence. Denote \(R' = \left[\begin{smallmatrix}\cos(\theta')&-\sin(\theta')\\\sin(\theta')&\cos(\theta')\end{smallmatrix}\right]\). Note that we have
\[\frac{\sqrt{n_j}}{q_j} = 
    \frac{\sqrt{aq_j\tilde{q}_j + O(1)}}{q_j} = 
    \sqrt{a\cdot\frac{\tilde{q}_j}{q_j} + O\biggl(\frac{1}{q_j}\biggr)} = 
    \sqrt{\frac{\beta}{4t^2\beta} + o(1)} = \frac{1}{2t} + o(1),\]
therefore \(q_j = 2t\sqrt{n_j} + o(\sqrt{n_j})\). We can now compute \(x_{n_j+q_j} - x_{n_j}\):
\begin{align*}
    &\mkern-36mu \Re e^{-2\pi i\alpha n_j}(x_{n_j+q_j} - x_{n_j})
    = \Re \sqrt{n_j+q_j} e^{2\pi i\alpha q_j} - \sqrt{n_j} \\
    &= \sqrt{n_j+q_j}\biggl(1 - O\biggl(\frac{1}{{q_j}^2}\biggr)\biggr) - \sqrt{n_j} \\
    &= \sqrt{n_j+q_j} - \sqrt{n_j} + \sqrt{\frac{n_j}{{q_j}^2} + \frac{1}{q_j}}\cdot o(1) \\
    &= t + o(1) + \biggl(\frac{1}{2t} + o(1)\biggr)\cdot o(1) = t + o(1) \\
    &\mkern-36mu \Im e^{-2\pi i\alpha n_j}(x_{n_j+q_j} - x_{n_j})
    = \Im \sqrt{n_j+q_j} e^{2\pi i\alpha q_j} - \sqrt{n_j} \\
    &= \sqrt{n_j+q_j}\sin(2\pi\alpha q_j) = \sqrt{n_j+q_j} \biggl(2\pi(q_j\alpha-p_j) + O\biggl(\frac{1}{{q_j}^3}\biggr)\biggr) \\
    &= \sqrt{\frac{n_j}{{q_j}^2} + \frac{1}{q_j}} \biggl(2\pi c + O\biggl(\frac{1}{{q_j}^2}\biggr)\biggr)
    = \biggl(\frac{1}{2t} + o(1)\biggr) (2\pi c + o(1)) = \frac{\pi c}{t} + o(1) ,
\end{align*}
therefore
\[x_{n_j+q_j} - x_{n_j} \to R' \begin{bmatrix} t \\ \pi c/t\end{bmatrix} = v' ,\]
and similarly
\[x_{n_j+\tilde{q}_j} - x_{n_j} \to R' \begin{bmatrix} \beta t \\ \pi \tilde{c}/(\beta t)\end{bmatrix} = \tilde{v}' .\]

This means that there is a converging subsequence of \(T_j'X\) in the Chabauty topology, where \(T_j'(x) = x - x_{n_j}\), and the limit must be a lattice that a sublattice that is generated by \(v'\) and \(\tilde{v}'\). By \Cref{prop:co_volume_of_limit_lattice} the limit lattice has a co-volume of \(\pi\), and by the same calculations, this is the co-volume of the lattice generated by \(v'\) and \(\tilde{v}'\), therefore the limit lattice is generated by \(v'\) and \(\tilde{v}'\). We proved that the lattice
\[\Lambda = R' \begin{bmatrix} t & \beta t \\ \pi c/t & \pi\tilde{c}/(\beta t)\end{bmatrix} =
    \sqrt{\pi} R' A_{t/\sqrt{\pi}} \begin{bmatrix} 1 & \beta \\ c & \tilde{c}/\beta\end{bmatrix} \]
is achieved as a Chabauty limit of \(X\).

Fix
\[\Lambda = \sqrt{\pi} R A_{t/\sqrt{\pi}} \begin{bmatrix} 1 & \beta \\ c & \tilde{c}/\beta\end{bmatrix} ,\]
where \(R\) is a rotation matrix. Consider \(n_j + b\) for some positive integer \(b\). It is clear from the calculations above that the Chabauty limit \(\Lambda^b\) of \(T_j^bX\), where \(T_j^b(x) = x - x_{n_j+b}\), is the lattice spanned by
\begin{align*}
    \Lambda^b = \sqrt{\pi} \begin{bmatrix}\cos(2\pi \alpha b) & -\sin(2\pi \alpha b) \\
        \sin(2\pi \alpha b) & \cos(2\pi \alpha b)\end{bmatrix} R' A_{t/\sqrt{\pi}} 
    \begin{bmatrix} t/\sqrt{\pi} &\beta t/\sqrt{\pi} \\ c \sqrt{\pi}/t & \tilde{c}/\beta \cdot \sqrt{\pi}/t\end{bmatrix}
\end{align*}
It is also clear that \(e^{2\pi i\alpha b}\) is dense in \(\mathbb{S}^1\) since \(\alpha\) is irrational, which concludes the proof since the set of Chabauty limits is closed.
\end{proof}

\section*{Further Questions}
The author conjectures that there is a spiral set that is a dense forest, this is in contrast to a conjecture of Adiceam and Tsokanos in \cite{Adiceam_2022}. A possible avenue to consider is noting that if all of the Chabauty limits of a set are dense forests, the set is also a dense forest (by \Cref{prop:dense_forest_chabauty}), and then constructing a spiral set whose limits are similar to the construction given by Bishop in \cite{Bishop_2011} using the lattice limits given by Fermat spirals. If there are indeed spiral sets that are dense forests, one might ask what is the best visibility function, and in particular whether there is a Danzer spiral set.

We saw a classification of all of the lattices that can be achieved, however, this classification is quite difficult to use. In particular, the question whether every lattice with co-volume \(\pi\) is achieved as a Chabauty limit of some Fermat spiral with badly-approximable angle difference remains.

A third question is what are the Chabauty limits of Fermat spirals with irrational angle difference that is not badly-approximable. As noted above, the author conjectures that the limits are isomorphic to \(\emptyset\), \(\mathbb{R}\), \(\mathbb{Z}^2\), and \(\mathbb{Z}\times\mathbb{R}\). Computational evidence suggests they are all achieved, but a full classification could be interesting, and in particular do the lattices also have co-volume \(\pi\).

\section*{Acknowledgement}
This paper is part of the Author's M.Sc. thesis at Tel Aviv University, conducted under the supervision of Barak Weiss. The financial support of grant ISF-NSFC 3739/21 is gratefully acknowledged.
\printbibliography

\end{document}